\documentclass[11pt]{amsart}
\usepackage{graphicx} 

\usepackage[utf8]{inputenc}
\usepackage[margin=1in,letterpaper,portrait]{geometry}
\usepackage{ytableau}
\ytableausetup{notabloids}
\ytableausetup{centertableaux}
\usepackage{latexsym,amsmath,amssymb,verbatim, tikz}
\usepackage{graphicx}
\usepackage{hyperref}

   \newtheorem{theorem}{Theorem}[section]
   \newtheorem{proposition}[theorem]{Proposition}
   
   \newtheorem{lemma}[theorem]{Lemma}
   \newtheorem{corollary}[theorem]{Corollary}

   \newtheorem{question}[theorem]{Question}

\numberwithin{equation}{section}

\newcommand{\todo}[1]{\vspace{5 mm}\par \noindent
\marginpar{\textsc{ToDo}} \framebox{\begin{minipage}[c]{0.95
\textwidth}
 #1 \end{minipage}}\vspace{5 mm}\par}

\newcommand{\la}{\lambda}

\newlength{\mysizetiny}
\newlength{\mysizesmall}
\newlength{\mysize}
\newlength{\mysizelarge}
\dedicatory{Dedicated to the memory of Avi Glick}

\title{A note on polynomials for character degrees}
\author{Yuval Roichman}
\address{Department of Mathematics, Bar-Ilan University, Ramat-Gan 52900, Israel}
\email{yuvalr@math.biu.ac.il}
\date{Sep 9, 2025}

\begin{document}

\maketitle

\begin{abstract}
A recent result of Cohen and Zemel provides an elegant expansion of 
 the Rasala polynomials for symmetric group character degrees.
In this note we present an alternative short algebraic proof. Extensions to polynomials of character values follow.  

\end{abstract}

\section{Introduction}

Fix a partition $\la=(\la_1,\dots,\la_t)\vdash k$ and let $n\ge \la_1$. Denote $(n,\la):=(n,\la_1,\dots,\la_t)\vdash n+k$. 
A classical theorem of Rasala~\cite{Rasala} expands the degree of $\chi^{(n,\la)}$ as a polynomial in $n$. The following elegant expansion was recently proved by Cohen and Zemel. 

Denote the degrees of the symmetric group characters 
$f^\la:
=\chi^{\la}(id)$ and  
$f^{\la/\mu}:=\chi^{\la/\mu}(id)$. 


\begin{theorem}\label{thm:CZ}~\cite{CZ} 
  For every fixed partition 
$\lambda=(\la_1,\dots,\la_t)\vdash k$ and 
 $n\ge \la_1$ the following holds. 
\[
f^{(n,\la)}=\sum\limits_{j=0}^t (-1)^j \binom{n+k}{k-j} f^{\la/(1^j)}.
\]   
\end{theorem}

The proof of Cohen and Zemel is combinatorial, using enumeration of standard Young tableaux, and sophisticated. A  symmetric function identity which implies Theorem~\ref{thm:CZ} and  
generalizations 
is presented in this paper. 

\section{A symmetric functions identity and Cohen-Zemel formula}

Recall the (skew) Schur function $s_\la$ ($s_{\la/\mu}$) indexed by a  (skew) partition $\la$ ($\la/\mu$).  

\begin{theorem}\label{thm:main}
For every fixed partition 
$\lambda=(\la_1,\dots,\la_t)\vdash k$ and 
 $n\ge \la_1$ the following holds. 
\begin{equation}\label{eq:main}
s_{(n,\la)}=\sum\limits_{j=0}^t (-1)^j s_{n+j}s_{\la/(1^j)}.  
\end{equation}
\end{theorem}

\begin{proof}
Let $\la/\mu$ be a skew shape and $\ell(\la)$ be the number of parts of $\la$.    
Recall the Jacobi-Trudi determinantal formula for skew Schur functions~\cite[Theorem 7.16.1]{EC2} 
\[
s_{\la/\mu} = \det H^{\la/\mu}
\]
where $H^{\la/\mu}$ is an $\ell(\la)\times \ell(\la)$  matrix whose entries are
\[
H^{\la/\mu}_{i,j}:= h_{\la_i-i-\mu_j+j}. 
\]
Expanding the determinant of $H^{(n,\la)}$ 
along the first row, one obtains 
\[
s_{(n,\la)}= \det H^{(n,\la)}=\sum\limits_{j=0}^t (-1)^j h_{n+j} M^{\la}_{1,j}=\sum\limits_{j=0}^t (-1)^j s_{n+j} M^{\la}_{1,j},   
\]
where $M^{\la}_{1,j}$ is the $(1,j)$-th minor of $H^{\la}$.

To complete the proof notice that, by the Jacobi-Trudi formula for skew shapes, 
\[
s_{\la/(1^j)} = M^{\la}_{1,j}.
\]
\end{proof}


\begin{proof}[Proof of Theorem~\ref{thm:CZ}.] 
Recall the Frobenius characteristic map~\cite[Ch. 7.18]{EC2}. 
Equation~\eqref{eq:main} is equivalent, via 
the Frobenius characteristic map, to  the following.
\begin{equation}\label{eq:ch}
\chi^{(n,\la)}=\sum\limits_{j=0}^t (-1)^j \left(\chi^{n+j}\otimes \chi^{\la/(1^j)}\right)\uparrow_{S_{n+j}\times S_{k-j}}^{S_n}. 
\end{equation}
Evaluating the characters at the identity and recalling that for every pair of skew shapes $\nu\vdash m$ and $\tau\vdash n-m$
\[
\left(\chi^{\nu}\otimes \chi^{\tau}\right)\uparrow_{S_{m}\times S_{n-m}}^{S_n}(id)=\binom{n}{m}f^\nu f^\tau, 
\]
complete the proof.  
\end{proof}


\section{Polynomiality of character values}


Recall the notation  $\alpha=(1^{a_1},2^{a_2},\dots,n^{a_n})\vdash n$ for a partition of $n$, with  $a_i$ parts of size $i$, for every $1\le i\le n$. Denote the evaluation of an $S_n$-character $\chi$ at a conjugacy class of cycle type $\alpha$ by  $\chi(\alpha)=\chi(1^{a_1},2^{a_2},\dots,n^{a_n})$; this notation is  sometimes abbreviated to $\chi_\alpha:=\chi(\alpha)$.  
The following lemma is an immediate consequence of the Murnaghan-Nakayma rule. For self-containment we give an independent basic proof. 

\begin{lemma}\label{prop:ind} 
For every pair 
of an 
$S_m$-character $\psi$  and an $S_{n-m}$-character   
$\phi$ 
\[
\left(\psi\otimes \phi\right)\uparrow_{S_{m}\times S_{n-m}}^{S_n}(1^{a_1},2^{a_2},\dots,n^{a_n})=
\sum\limits_{(1^{b_1},\dots,n^{b_n})\vdash m\atop \forall i\ 0\le b_i\le a_i}
\prod\limits_{i=1}^n \binom{a_i}{b_i}  
\psi{(1^{b_1},\dots,n^{b_n})} \phi {(1^{a_1-b_1},\dots,n^{a_n-b_n})}. 
\]
\end{lemma}


\begin{proof}
The following formula for induced characters is well known, see e.g.~~\cite[Ch. 3.3 Thm. 12]{Serre}. Let $H$ be  
subgroup of a finite group  $G$, and let 
$\chi$ be an $H$-character. Then  
\[
\chi\uparrow_H^G(x)=\frac{1}{|H|}
\sum\limits_{g\in G\atop gxg^{-1}\in H}\chi(gxg^{-1})=
\frac{|{\text{Cent}}_G(x)|}{|H|}\cdot \sum\limits_{gxg^{-1}\in H}
\chi(gxg^{-1}) 
\]
To verify the RHS observe that the number of elements in $G$ that conjugate $x$ to any fixed $gxg^{-1}$ is equal to the size of the centralizer ${\text{Cent}}_G(x)$. Let $R_x\subset H$ be a set of representatives of conjugacy classes in $H$ which are conjugate to $x$ in $G$.  Then 
\[
\chi\uparrow_H^G(x)= \frac{|{\text{Cent}}_G(x)|}{|H|}\cdot \sum\limits_{r\in R_x} \frac{|H|}{|{\text{Cent}}_H(r)|} 
\chi(r) =
 \sum\limits_{r\in R_x}
\frac{|{\text{Cent}}_G(x)|}{|{\text{Cent}}_H(r)|}
\chi(r).
\]
For a permutation $\pi\in S_n$ of cycle type $\alpha=(1^{a_1},\dots,n^{a_n})\vdash n$, 
all conjugates of $\pi$ in $S_m\times S_{n-m}$ 
are products 
$\pi_1\times \pi_2$ in the Young subgroup $S_m\times S_{n-m}$, such that 
$\pi_1$ has cycle type  $(1^{b_1},\dots,n^{b_n})\vdash m$ and $\pi_2$ has cycle type $(1^{a_1-b_1},\dots,n^{a_n-b_n})\vdash n-m$,  for some $0\le b_i\le a_i$  $(\forall i)$. Then    
\[
\frac{|{\text{Cent}}_{S_n}(1^{a_1},\dots,n^{a_n})|}{|{\text{Cent}}_{S_m\times S_{n-m}}\left((1^{b_1},\dots,n^{b_n})\times  
(1^{a_1-b_1},\dots,n^{a_n-b_n})\right)|}
\
=\frac{\prod\limits_{i=1}^n a_i! i^{a_i}}{
\prod\limits_{i=1}^n b_i! i^{b_i}\times \prod\limits_{i=1}^n (a_i-b_i)! i^{a_i-b_i}}=\prod\limits_{i=1}^n \binom{a_j}{b_j}.
\]
\end{proof}


Theorem~\ref{thm:CZ} may be generalized to all conjugacy classes of cycle type $r^s$ as follows. 

\begin{proposition}\label{cor:r} 
For every 
fixed partition $\la=(\la_1,\dots,\la_t)\vdash k$, $n\ge \la_1$ and 
 $r|(n+k)$ the following holds. 
 \begin{equation}\label{eq:r}
\chi^{(n,\la)}_{r^{\frac{n+k}{r}}}
=\sum\limits_{j=0}^t (-1)^j \binom{\frac{n+k}{r}}{\frac{k-j}{r}} \chi^{\la/(1^j)}_{r^{\frac{k-j}{r}}},  
\end{equation}
where the sum runs over $j$ with $r|(k-j)$. 
In particular, $\chi^{(n,\la)}_{r^{\frac{n+k}{r}}}$ is a polynomial in $n$.
\end{proposition}


\begin{proof}
    Let $\pi\in S_{n+k}$ be a   permutation  of cycle type ${r^{\frac{n+k}{r}}}$.  
    By Equation~\eqref{eq:ch},  
    \[
\chi^{(n,\la)}(\pi)=\sum\limits_{j=0}^t (-1)^j \left(\chi^{n+j}\otimes \chi^{\la/(1^j)}\right)\uparrow_{S_{n+j}\times S_{k-j}}^{S_n}(\pi).
\]
Letting $\phi=\chi^{(n+j})$, $\psi=\chi^{\la/(1^j)}$ and 
$\alpha=(r^{\frac{n+k}{r}})$ in Lemma~\ref{prop:ind},  the RHS is equal to 
\[
\sum\limits_{j=0}^t (-1)^j \binom{\frac{n+k}{r}}{\frac{k-j}{r}} \chi^{\la/(1^j)}_{r^{\frac{k-j}{r}}},  
\]
where the sum runs over $j$ with $r|(k-j)$, completing the proof of Equation~\eqref{eq:r}.

Finally, 
the binomial coefficients 
$\binom{\frac{n+k}{r}}{\frac{k-j}{r}}$ are polynomials in $n$. 
The character values in the sum depend on fixed parameters only,   
and the number of summands is bounded, completing the proof that the RHS of Equation~\eqref{eq:r} is a polynomial in $n$. 
\end{proof}



\smallskip


Theorem~\ref{thm:main} further implies another generalization of Rasala's result.

\smallskip 



For $\nu=(1^{a_1},2^{a_2},\dots,m^{a_m})\vdash m$ denote $(\nu,1^d):=(1^{a_1+d},2^{a_2},\dots,m^{a_m})\vdash m+d$. 

\begin{proposition}
For every fixed $\la\vdash k$ and $\nu\vdash m$,  
the character value 
\[
\chi^{(n,\la)}_{(\nu,1^{n+k-m})}
\]
is a polynomial in $n$.
\end{proposition}


\begin{proof}
By Theorem~\ref{thm:main} together with 
Lemma~\ref{prop:ind}, for any $\nu=(1^{a_1},2^{a_2},\dots,m^{a_m})$ 
\[
\chi^{(n,\la)}_{(\nu,1^{n+k-m})}= 
\sum\limits_{j=0}^t (-1)^j
\left(\chi^{\la/(1^j)}\otimes \chi^{(n+j)}\right)\uparrow_{S_{k-j}\times S_{n+j}}^{S_{n+k}}(1^{n+k-m+a_1},2^{a_2},\dots,m^{a_m})
\]
\[
=
\sum\limits_{j=0}^t (-1)^j
\sum\limits_{0\le b_1\le 
\min\{k-j,n+k-m\}} \binom{n+k-m+a_1}{b_1}
\sum\limits_{(2^{b_2},\dots,m^{b_m})\vdash k-j-b_1\atop \forall i>1\ 0\le b_i\le a_i} 
\prod\limits_{i=2}^m \binom{a_i}{b_i}
\chi^{\la/(1^j)}{(1^{b_1},\dots,m^{b_m})}.  
\]
Clearly, 
the binomial coefficients $\binom{n+k-m+a_1}{b_1}$ 
are polynomials in $n$ for every $b_1\le k-j$,
 since $m$  and $k$ are fixed. 
All other binomial coefficients  and character values  
in the formula depend on fixed parameters only and the number of the summands is bounded. 
Proof is completed. 
\end{proof}



\begin{question} 
Find explicit nice expansions for such (and similar) polynomials of character values.   
\end{question}

This 
problem was recently studied by Moshayov and Zemel~\cite{MZ}.

\end{document}